\documentclass[a4paper]{amsart}
\numberwithin{equation}{section}
\newtheorem{theorem}{Theorem}[section]

\newtheorem{definition}{Definition}[section]

\theoremstyle{remark}

\usepackage{amsmath, hyperref}
\usepackage{color}
\usepackage{graphicx}
\usepackage{subfigure}

\title[Notes on the norm of pre-Schwarzian derivatives]{Notes on the norm of pre-Schwarzian derivatives on bi-univalent functions of order $\alpha$}



\begin{document}
\begin{abstract}
In the present paper we estimate the norm of the pre-Schwarzian derivative of bi-starlike functions of order $\alpha$ where $\alpha\in[0,1)$.
Initially this problem was handled by Rahmatan et al. in [Bull Iran Math Soc {\bf43}: 1037-1043, 2017]. We pointed out that the proofs and bounds by Rahmatan et al. are incorrect and present correct proofs and bounds.
\end{abstract}

\author[H. Mahzoon and R. Kargar] {H. Mahzoon and R. Kargar}
\address{Department of Mathematics, Islamic Azad University, West Tehran
Branch, Tehran, Iran}
\email {hesammahzoon1@gmail.com}
\address{Young Researchers and Elite Club,
Ardabil Branch, Islamic Azad University, Ardabil, Iran}
       \email{rkargar1983@gmail.com}

\maketitle

\section{Introduction}
Let $\mathcal{H}$ be the family of analytic functions $f$ on the open unit disc $\Delta=\{z\in \mathbb{C}:|z|<1\}$ and $\mathcal{A}$ be a subclass of $\mathcal{H}$ that its members are normalized by the condition $f(0)=0=f'(0)-1$. This means that each $f\in\mathcal{A}$ has the following form
\begin{equation}\label{f}
f(z)=z+ \sum_{n=2}^{\infty}a_{n}z^{n}\quad(z\in\Delta).
\end{equation}
We denote by $\mathcal{U}\subset\mathcal{A}$ the family of univalent (one--to--one) functions $f$ in $\Delta$. Let $f$ and $g$ belong to the class $\mathcal{H}$. A function $f$ is called to be {\it subordinate} to $g$, written as
\begin{equation*}
  f(z)\prec g(z)\quad{\rm or}\quad f\prec g,
\end{equation*}
if there exists a Schwarz function $\phi:\Delta\rightarrow\Delta$ with
the following properties
\begin{equation*}
  \phi(0)=0\quad{\rm and}\quad |\phi(z)|<1\quad(z\in\Delta),
\end{equation*}
such that $f(z)=g(\phi(z))$ for all $z\in\Delta$. If
$g\in\mathcal{U}$, then the following geometric equivalence
relation holds true:
\begin{equation*}
  f(z)\prec g(z)\Leftrightarrow f(0)=g(0)\quad{\rm and}\quad f(\Delta)\subset g(\Delta).
\end{equation*}

It is known that the {\it Koebe} function
\begin{equation*}
  k(z):=\frac{z}{(1-z)^2}=z+2z^2+3z^3+\cdots+nz^n+\cdots\quad(z\in\Delta),
\end{equation*}
maps the open unit disc $\Delta$ onto the entire plane minus the interval $(-\infty,-1/4]$.
Also, the well-known {\it Koebe One-Quarter Theorem} states that the image of the open unit disc $\Delta$ under every function $f\in\mathcal{U}$ contains the disc $\{w:|w|<\frac{1}{4}\}$, see \cite[Theorem 2.3]{Duren}. Therefore every function $f$ in the class $\mathcal{U}$ has an inverse $f^{-1}$ which satisfies the following conditions:
\begin{equation*}
  f^{-1}(f(z))=z\quad(z\in\Delta)
\end{equation*}
and
\begin{equation*}
  f(f^{-1}(w))=w\quad(|w|<r_0(f);\ \ r_0(f)\geq1/4),
\end{equation*}
where
\begin{equation}\label{f inverse}
  f^{-1}(w)=w-a_2w^2+(2a_2^2-a_3)w^3-(5a_2^3-5a_2a_3+a_4)w^4+\cdots=:g(w).
\end{equation}
A function $f\in\mathcal{A}$ is {\it bi-univalent} in $\Delta$ if, and only if, both $f$ and $f^{-1}$ are univalent in $\Delta$. We denote by $\Sigma$ the class of bi-univalent functions in $\Delta$. The functions
\begin{equation*}
  f_1(z)=\frac{z}{1-z},\quad f_2(z)=-\log(1-z)\quad {\rm and} \quad f_3(z)=\frac{1}{2}\log\left(\frac{1+z}{1-z}\right),
\end{equation*}
with the corresponding inverse functions, respectively
\begin{equation*}
  f_1^{-1}(w)=\frac{w}{1+w},\quad f_2^{-1}(w)=\frac{\exp(w)-1}{\exp(w)}\quad {\rm and} \quad f_3^{-1}(w)=\frac{\exp(2w)-1}{\exp(2w)+1},
\end{equation*}
belong to the class $\Sigma$.

A function $f\in\mathcal{A}$ is called starlike (with respect to $0$) if $tw\in f(\Delta)$ whenever $w\in f(\Delta)$ and $t\in[0, 1]$. We denote by $\mathcal{S}^*$ the class of the starlike functions in $\Delta$. Also, we say that a function $f\in\mathcal{A}$ is starlike of order $\alpha$ ($0\leq\alpha<1$) if, and only if,
\begin{equation}\label{zf prime f starlike}
  {\rm Re}\left\{\frac{zf'(z)}{f(z)}\right\}>\alpha\quad(z\in\Delta).
\end{equation}
The class of the starlike functions of order $\alpha$ in $\Delta$ is denoted by
$\mathcal{S}^*(\alpha)$. Now by definition of the starlike functions of order $\alpha$ we recall the following definition.
\begin{definition}\label{Def1}
 Let $\alpha\in [0,1)$. A function $f\in\Sigma$ given by \eqref{f} is said to be in the class $\mathcal{S}^*_\Sigma(\alpha)$ if the conditions \eqref{zf prime f starlike} and
\begin{equation}\label{wf prime f starlike}
  {\rm Re}\left\{\frac{wg'(w)}{g(w)}\right\}>\alpha\quad(w\in\Delta),
\end{equation}
hold true where $w=f(z)$ and $g=f^{-1}$.
\end{definition}
Denote by $\mathcal{V}(\alpha)$ the class of all $f\in\mathcal{A}$ satisfying the following condition
\begin{equation}\label{V alpha}
  {\rm Re}\left\{\left(\frac{z}{f(z)}\right)^2 f'(z)\right\}>\alpha\quad(0\leq \alpha<1, z\in\Delta).
\end{equation}
The class $\mathcal{V}(\alpha)$ was considered by Obradovi\'{c} and Ponnusamy in \cite{ObPo}. Now we recall the following definition.
\begin{definition}\label{Def2}
 Let $\alpha\in [0,1)$. A function $f\in\Sigma$ given by \eqref{f} is said to be in the class $\mathcal{V}_\Sigma(\alpha)$ if the conditions \eqref{V alpha} and
\begin{equation}\label{w g(w)}
  {\rm Re}\left\{\left(\frac{w}{g(w)}\right)^2 g'(w)\right\}>\alpha\quad(w\in\Delta),
\end{equation}
hold true where $w=f(z)$ and $g=f^{-1}$.
\end{definition}

Denote by $\mathcal{LU}$ the subclass of $\mathcal{H}$
consisting of all locally univalent functions, namely, $\mathcal{LU}:=\{f\in\mathcal{H}: f'(z)\neq 0,\ \ z\in\Delta\}$. For a $f\in\mathcal{LU}$ the pre--Schwarzian and Schwarzian derivatives of $f$ are defined as follows
\begin{equation*}
  T_f(z):=\frac{f''(z)}{f'(z)}\quad{\rm and}\quad S_f(z):=T'_f(z)-\frac{1}{2}T^2_f(z),
\end{equation*}
respectively. We note that the quantity $T_f$ (resp. $S_f$) is analytic on $\Delta$ precisely when $f$ is analytic (resp. meromorphic) and locally univalent on $\Delta$. Since $\mathcal{LU}$ is a vector space over $\mathbb{C}$ (see \cite{Hornich}), thus we can define the norm of $f\in\mathcal{LU}$ by
\begin{equation*}
  ||f||=\sup_{z\in\Delta}(1-|z|^2)\left|\frac{f''(z)}{f'(z)}\right|.
\end{equation*}
This norm has significance in the theory of Teichm\"{u}ller spaces, see \cite{Astala}. It is known that $||f||<\infty$ if and only if $f$ is uniformly locally univalent. Also, notice that if $||f||\leq1$, then $f$ is univalent in $\Delta$ and conversely, if $f$ univalent in $\Delta$, then $||f||\leq 6$ and the equality is attained for the Koebe function and its rotation. In fact, both of these bounds are sharp, see \cite{BecPom}. Also, if $f$ is starlike of order $\alpha\in[0,1)$, then we have the sharp estimate $||f||\leq 6-4\alpha$ (see e.g. \cite{Yamashita}).

The following theorems were wrongly proven by Rahmatan et al. \cite[Theorem 2.2 and Theorem 2.4]{Rahmatan2017}.\\
{\bf Theorem A.} Let the function $f(z)$ given by \eqref{f} be in the class $\mathcal{S}^*_\Sigma(\alpha)$ where $\alpha\in[0,1)$. Then
\begin{equation}\label{norm f Rahm S}
  ||f||\leq \min\{6-4\alpha, 4\alpha+2\}.
\end{equation}
{\bf Theorem B.} Let the function $f(z)$ given by \eqref{f} be in the class $\mathcal{V}_\Sigma(\alpha)$ where $\alpha\in[0,1)$. Then
\begin{equation}\label{norm f Rahm V}
  ||f||\leq \min\{10-8\alpha, 6-8\alpha\}.
\end{equation}
We notice that $6-8\alpha\in (-2,6]$ when $\alpha\in[0,1)$.

In the present paper we give a correct proof for Theorem A and some remarks on the Theorem B.
\section{On the Theorem A}
The correct version of Theorem A is contained in the following theorem:
\begin{theorem}\label{t1}
Let $\alpha\in[0,1)$.
  If the function $f$ of the form \eqref{f} is in the class $\mathcal{S}^*_\Sigma(\alpha)$, then
\begin{equation}\label{norm f for starlike singma}
  ||f||\leq \left\{
              \begin{array}{ll}
                6, & \hbox{$\alpha=0$;} \\
                \min\left\{6-4\alpha,\frac{4(1-\alpha)}{\alpha}\right\}, & \hbox{$0< \alpha< \frac{1}{2}$;} \\
                4, & \hbox{$\alpha=1/2$;} \\
                2, & \hbox{$1/2<\alpha<1$.}
              \end{array}
            \right.
\end{equation}
\end{theorem}
\begin{proof}
  Let $f\in \mathcal{S}^*_\Sigma(\alpha)$ be given by \eqref{f}. Then by Definition \ref{Def1} the following conditions hold true:
\begin{equation}\label{star p1}
  {\rm Re}\left\{\frac{zf'(z)}{f(z)}\right\}>\alpha\quad(z\in\Delta)
\end{equation}
  and
\begin{equation}\label{star p2}
  {\rm Re}\left\{\frac{wg'(w)}{g(w)}\right\}>\alpha\quad(w\in\Delta),
\end{equation}
where $w=f(z)$ and $g=f^{-1}$. Since \eqref{star p1} holds true by \cite{Yamashita} we conclude that $||f||\leq 6-4\alpha$ where $\alpha\in [0,1)$. On the other hand, since the mapping
\begin{equation}\label{F}
  F(w):=\frac{1+(1-2\alpha)w}{1-w}\quad(w\in\Delta)
\end{equation}
maps $\Delta$ onto the right half-plane having real part greater than $\alpha$, thus \eqref{star p2} implies that
\begin{equation}\label{w g prime sub F}
  \frac{wg'(w)}{g(w)}\prec F(w)\quad(w\in\Delta).
\end{equation}
By the identity $\frac{{\rm d}}{{\rm d}w}f^{-1}(w)=\frac{1}{f'(z)}$ (see \cite[Eq. (1.2)]{Rahmatan2017}) \eqref{w g prime sub F} is equivalent to
\begin{equation}\label{star p3}
\frac{f(z)}{zf'(z)}\prec\frac{1+(1-2\alpha)z}{1-z}\quad(z\in\Delta).
\end{equation}
 Now by \eqref{star p3} and definition of subordination there exists a Schwarz function $\phi$ with $\phi(0)=0$ and $|\phi(z)|<1$ $(z\in\Delta)$ such that
\begin{equation*}
  \frac{f(z)}{zf'(z)}=\frac{1+(1-2\alpha)\phi(z)}{1-\phi(z)}\quad(z\in\Delta),
\end{equation*}
 or equivalently
 \begin{equation}\label{star p4}
  \frac{f'(z)}{f(z)}=
  \frac{1-\phi(z)}{z(1+(1-2\alpha)\phi(z))}\quad(0\neq z\in\Delta).
\end{equation}
Taking logarithm on both sides of \eqref{star p4} and differentiating, after simplification, we obtain
\begin{equation}\label{f prime2}
  \frac{f''(z)}{f'(z)}=\frac{2(\alpha-1)\phi(z)}{z(1+(1-2\alpha)\phi(z))}
  +\frac{2(\alpha-1)\phi'(z)}{(1-\phi(z))(1+(1-2\alpha)\phi(z))}\quad
  (0\neq z\in\Delta).
\end{equation}
The well-known Schwarz-Pick lemma states that for a Schwarz function $\phi$ the following inequality holds true:
\begin{equation}\label{Sch-Pick}
  |\phi'(z)|\leq \frac{1-|\phi(z)|^2}{1-|z|^2}\quad(z\in\Delta).
\end{equation}
Now by \eqref{Sch-Pick} and since $|\phi(z)|\leq |z|$ (see \cite{Duren}), the relation \eqref{f prime2} implies that
\begin{align*}
  \left|\frac{f''(z)}{f'(z)}\right| &=\left|\frac{2(\alpha-1)\phi(z)}{z(1+(1-2\alpha)\phi(z))}
  +\frac{2(\alpha-1)\phi'(z)}{(1-\phi(z))(1+(1-2\alpha)\phi(z))}\right| \\
  &\leq \frac{2|\alpha-1||\phi(z)|}{|z|(1-|1-2\alpha||\phi(z)|)}
  +\frac{2|\alpha-1||\phi'(z)|}{(1-|\phi(z)|)(1-|1-2\alpha||\phi(z)|)}\\
  &\leq \frac{2(1-\alpha)}{1-|1-2\alpha||z|}+\frac{2(1-\alpha)}{1-|1-2\alpha||z|}.
  \frac{1+|z|}{1-|z|^2}.
\end{align*}
However we get
\begin{align*}
  ||f||&=\sup_{z\in\Delta}(1-|z|^2)\left|\frac{f''(z)}{f'(z)}\right| \\
  &\leq\sup_{z\in\Delta}\left\{\frac{2(1-\alpha)(1-|z|^2)}{1-|1-2\alpha||z|}
  +\frac{2(1-\alpha)(1+|z|)}{1-|1-2\alpha||z|}\right\}\\
  &=\frac{4(1-\alpha)}{1-|1-2\alpha|}=:\varphi(\alpha).
\end{align*}
Now we consider the following cases.\\
{\bf Case 1.} $\alpha=0,1/2$. In this case we have $\varphi(\alpha)=\infty$. Therefore 
\begin{equation*}
\min\{6-4\alpha,\varphi(\alpha)\}=6-4\alpha.
\end{equation*}
{\bf Case 2.} $0<\alpha<1/2$. In this case $\varphi(\alpha)$ becomes $\frac{2(1-\alpha)}{\alpha}\in(2,\infty)$. Also we have $6-4\alpha\in(2,4)$ when $\alpha\in(1/2,1)$. Thus 
\begin{equation*}
||f||\leq \min\left\{6-4\alpha,\frac{4(1-\alpha)}{\alpha}\right\}.
\end{equation*}
{\bf Case 3.} $1/2<\alpha<1$. In this case we have $\varphi(\alpha)=2$. Thus
\begin{equation*}
\min\{6-4\alpha,\varphi(\alpha)\}=2.
\end{equation*}
Because $2<6-4\alpha$ when $\alpha\in(1/2,1)$. This completes the proof.
\end{proof}
\section{On the Theorem B}
In this section we give some remarks on the Theorem B. Let a function $f$ be of the form \eqref{f} belongs to the class $\mathcal{V}_\Sigma(\alpha)$. Then the conditions \eqref{V alpha} and \eqref{w g(w)} hold true. From the geometric meaning of the function $(1+(1-2\alpha)z)/(1-z)$, subordination principle and \eqref{V alpha} we have
\begin{equation*}
  \left(\frac{z}{f(z)}\right)^2 f'(z)\prec \frac{1+(1-2\alpha)z}{1-z}\quad(z\in\Delta).
\end{equation*}
Using definition of subordination there exists a Schwarz function $\phi$ so that
\begin{equation}\label{B p1}
  \left(\frac{z}{f(z)}\right)^2 f'(z)=\frac{1+(1-2\alpha)\phi(z)}{1-\phi(z)}\quad(z\in\Delta).
\end{equation}
With a simple calculation \eqref{B p1} yields
\begin{equation}\label{B p2}
  \frac{f''(z)}{f'(z)}=2\left(\frac{f'(z)}{f(z)}-\frac{1}{z}\right)
  +\frac{(1-2\alpha)\phi'(z)}{1+(1-2\alpha)\phi(z)}+\frac{\phi'(z)}{1-\phi(z)}
\quad(z\in\Delta).
\end{equation}
Again, from \eqref{B p1} we get
\begin{equation}\label{B p3}
  \frac{f'(z)}{f(z)}-\frac{1}{z}=\frac{f(z)}{z}\left(\frac{1+(1-2\alpha)\phi(z)}
  {z(1-\phi(z))}-1\right)\quad(z\in\Delta).
\end{equation}
By substituting \eqref{B p3} into \eqref{B p2}, we have
\begin{equation}\label{B p4}
  \frac{f''(z)}{f'(z)}=2\frac{f(z)}{z}\left(\frac{1+(1-2\alpha)\phi(z)}
  {z(1-\phi(z))}-1\right)
  +\frac{(1-2\alpha)\phi'(z)}{1+(1-2\alpha)\phi(z)}+\frac{\phi'(z)}{1-\phi(z)}
\quad(z\in\Delta).
\end{equation}
To estimate the above relationship \eqref{B p4}, we need to estimate $|f(z)/z|$, but so far this is an open question. So this question seems to be more difficult than what is given in \cite{Rahmatan2017}. Also, we remark that in \cite[proof of Theorem 2.4]{Rahmatan2017} Rahmatan et al. mistakenly used the following relation
\begin{equation*}
  \frac{f'(z)}{f(z)}=\frac{1+(1-2\alpha)\phi(z)}{z(1-\phi(z))}
\end{equation*}
and this means that $f\in \mathcal{S}^*(\alpha)$. Thus Theorem 2.4 and its proof in \cite{Rahmatan2017} is incorrect.


\end{document}